\documentclass[11pt,reqno]{amsart}

\usepackage{amssymb}
\usepackage{amsfonts}
\usepackage{amsthm}
\usepackage{amsmath}
\usepackage{bbm}
\usepackage{xcolor}
\usepackage{url}

\numberwithin{equation}{section}
\allowdisplaybreaks

\newtheorem{theorem}{Theorem}[section]
\newtheorem{proposition}[theorem]{Proposition}

\newtheorem{lemma}[theorem]{Lemma}

\theoremstyle{definition}

\theoremstyle{remark}
\newtheorem{remark}[theorem]{Remark}

\newcommand{\R}{\mathbb{R}}

\newcommand{\N}{\mathbb{N}}
\newcommand{\Z}{\mathbb{Z}}
\newcommand{\C}{\mathbb{C}}

\newcommand{\eps}{\varepsilon}

\newcommand{\scriptD}{\mathcal{D}}
\newcommand{\scriptE}{\mathcal{E}}

\newcommand{\scriptJ}{\mathcal{J}}
\newcommand{\scriptK}{\mathcal{K}}

\newcommand{\scriptS}{\mathcal{S}}
\newcommand{\scriptT}{\mathcal{T}}

\newcommand{\fE}{\mathcal{E}}

\newcommand{\qtq}[1]{\quad\text{#1}\quad}

\DeclareMathOperator*{\supp}{supp}
\DeclareMathOperator*{\diam}{diam}
\DeclareMathOperator*{\dist}{dist}

\begin{document}
\title[Extension maximizability from the hyperbolic paraboloid]{Existence of maximizers for $L^p$ Fourier extension from the hyperbolic paraboloid}

\author[Bruce]{Benjamin Bruce}
\author[Oliveira e Silva]{Diogo Oliveira e Silva}
\author[Stovall]{Betsy Stovall}

\address{Trinity College \\
300 Summit Street\\
Hartford, CT 06106\\
 USA.} 
\email{benjamin.bruce@trincoll.edu}

\address{  
Center for Mathematical Analysis, Geometry and Dynamical Systems \&
Instituto Superior T\'{e}cnico\\
Universidade de Lisboa\\
Av. Rovisco Pais\\ 
1049-001 Lisboa, Portugal.} 
\email{diogo.oliveira.e.silva@tecnico.ulisboa.pt}

\address{University of Wisconsin--Madison\\
480 Lincoln Drive\\ 
Madison, WI 53706\\
USA.}
\email{stovall@math.wisc.edu}



\begin{abstract}
    We prove that maximizers exist and that maximizing sequences possess subsequences that converge modulo symmetries to maximizers for the $L^p \to L^q$ Fourier extension inequalities associated to the hyperbolic paraboloid in three ambient dimensions.  
\end{abstract}
\maketitle


\section{Introduction}


Consider the hyperbolic paraboloid in $\R^3$,
\begin{equation}
    \Sigma:=\{(\xi_1\xi_2,\boldsymbol\xi)\in\R^3: \boldsymbol\xi=(\xi_1,\xi_2)\in\R^2\},
\end{equation}
which is a doubly ruled surface with negative gaussian curvature at every point. 
The corresponding Fourier extension operator,
\begin{equation} \label{E:extn}
\scriptE f(t,\boldsymbol x) := \int_{\R^2} e^{i(t,\boldsymbol x)\cdot(\xi_1\xi_2,\boldsymbol\xi)}f(\boldsymbol\xi)\, d\boldsymbol\xi, \quad (t,\boldsymbol x)\in\R^{1+2},
\end{equation}
is initially defined for smooth functions with compact support in $\R^2$.
The restriction conjecture predicts the validity of the estimate
\begin{equation}\label{eq_ExtEst}
    \|\mathcal E f\|_{L^{2p'}(\R^{1+2})} \leq C_p \|f\|_{L^p(\R^2)},
\end{equation}
for all exponents $1\leq p<3$, where $\frac1p+\frac1{p'}=1$.

The case $p=1$ of \eqref{eq_ExtEst} is elementary and the case $p=2$ goes back to  early investigations of Stein \cite{St93}, Tomas \cite{To75}, and Strichartz \cite{St77}. Following subsequent work by  Vargas \cite{Va05}, Lee \cite{Le06}, Cho--Lee \cite{CL17}, and Kim \cite{Ki17} on the local extension problem off the scaling line, Stovall \cite{BShypparab} established  estimate \eqref{eq_ExtEst}  whenever  $1\leq p<\frac{13}5$.
This was accomplished via a refined version of the bilinear-to-linear deduction of Lee and Vargas for extension estimates from the hyperbolic paraboloid $\Sigma$ to the sharp line, and constituted  the first scalable extension estimate for a negatively curved hypersurface beyond the Stein--Tomas range. More recently, Demeter--Wu \cite{DW25} and then Guo--Liu--Xi \cite{GLX26} proved bilinear estimates which, via the aforementioned bilinear-to-linear mechanism (\cite[Theorem 1.1]{BShypparab}), extend the known range of exponents for which estimate \eqref{eq_ExtEst} holds to $1\leq p<\frac83$  and $1\leq p<\frac{11}4$, respectively. The latter remains the current record for this problem.

Off the scaling line, restriction estimates  to perturbed  hyperbolic paraboloids have been considered by Buschenhenke--Müller--Vargas  \cite{BMV20}--\cite{BMV22} and Guo--Oh \cite{GO24}; see also \cite{GLX26}.  The higher dimensional case has been addressed by Hickman--Iliopoulou \cite{HI22}, Barron \cite{Ba22}, and Li \cite{Li25}.

We are interested in the sharp form of \eqref{eq_ExtEst}, and consider the operator norm
\[{\bf A}_p:=\sup_{\|f\|_{L^p(\R^2)}=1}{\|\mathcal Ef\|_{L^{2p'}(\R^{1+2})}},\]
which corresponds to the best constant $C_p$ for which \eqref{eq_ExtEst} holds. One easily checks that ${\bf A}_1=1$,  that any nonnegative $f$ maximizes the corresponding inequality, and that in this case maximizing sequences need not be compact.

\subsection{Main result}
Our main result addresses the precompactness of arbitrary maximizing sequences for the nonendpoint instances of inequality \eqref{eq_ExtEst} modulo symmetries and, in particular, the existence of maximizers.

\begin{theorem}\label{thm_main}
    Assume that $\mathcal E$ extends as a bounded linear operator from $L^{p_0}(\R^2)$ to $L^{2p_0'}(\R^{1+2})$, for some $1<p_0<3$. 
    Then, for all $1<p<p_0$, there exist nonzero functions $f\in L^p(\R^2)$, such that $\|\mathcal Ef\|_{2p'}={\bf A}_{p} \|f\|_p.$
    Furthermore, if $\{f_n\}\subseteq L^p(\R^2)$ is any norm-one sequence with $\lim_{n\to\infty}\|\mathcal E f_n\|_{2p'}={\bf A}_{p}$, then there exists a subsequence $\{f_{n_k}\}$ of $\{f_n\}$ and a sequence $\{S_k\}$ of symmetries of $\mathcal E$ such that $S_kf_{n_k}$ converges in $L^p(\R^2)$ to a maximizer of $\mathcal E$.
\end{theorem}
\noindent 
The symmetries of $\mathcal E$ to which we refer here are dilations, frequency translations, and modulations, and will be described precisely in \S\ref{sec_symmetries}.  In the range $\frac{11}4\leq p_0<3$ where the boundedness of the extension operator is still an open question, Theorem \ref{thm_main} is conditional on improvements towards the restriction conjecture for the hyperbolic paraboloid.

\subsection{Previous work}
There is a large body of work aiming at establishing the existence of maximizers in sharp restriction theory, but this is classically limited to the case $p=2$. In fact, whenever Hilbert space techniques (e.g.\@ Plancherel) are available, they turn out to be very powerful.
We refer to \cite{DMPS18, FVV12} for the present case of the hyperbolic paraboloid $\Sigma$, and to the surveys \cite{FOS17, NOST22} for other related but distinct $L^2$ instances.

More recently, the following results, valid for arbitrary exponents $p\neq 2$ in all admissible dimensions, have appeared in the literature.

\noindent --  All valid, scale-invariant, nonendpoint Fourier restriction inequalities to the {\bf paraboloid} \cite{BS-parab-ext-exist}, the {\bf cone} \cite{NOSSS25}, and the {\bf moment curve}\footnote{The more general case of monomial curves was subsequently treated in \cite{BS25}.} \cite{BS23}  have maximizers, and $L^p$-normalized maximizing sequences are precompact modulo symmetries. 

\noindent -- On the {\bf sphere} \cite{FS22}, precompactness of $L^p$-normalized maximizing sequences is ensured  off the parabolic scaling line, modulo spacetime modulations; on the parabolic scaling line, a   ``precompactness vs.\@ antipodal concentration'' dichotomy is established in the same paper.

The proofs follow the general principle established for the paraboloid in \cite{BS-parab-ext-exist} of first showing that maximizing sequences possess good frequency localization, and then using the $L^2$-theory to construct a profile decomposition of frequency-localized sequences with non-negligible extensions.
There are, however,  multiple twists and turns along the way, of which we highlight the following:

\begin{itemize}
\item[(\cite{NOSSS25})] The symmetry group of the cone is more complex than that of the paraboloid or sphere, leading to a two-step frequency localization.\smallskip
\item[(\cite{BS23, BS25})] Adapting the concentration-compactness methods to Fourier restriction on higher codimensional manifolds like the moment curve is a challenge which requires a multilinear generalization of the bilinear-to-linear argument of Tao--Vargas--Vega \cite{TVV98}.\smallskip
\item[(\cite{FS22})] On the sphere, the lack of scaling symmetry  and the existence of distinct points with parallel normal vectors are additional difficulties. Scaling should be treated as an almost-symmetry, and the fact that concentration at two points is better than concentration
at a single point is a non-local phenomenon that had been identified by Christ--Shao \cite{CS12} and explored by Frank--Lieb--Sabin \cite{FLS16} when $p=2$.
\end{itemize}

\subsection{What's new?}
None of the above proofs carries over easily to the hyperbolic paraboloid.
The major complication is tied to  the negative curvature of $\Sigma$.
Bilinear-to-linear reduction is much more indirect, and the refined Strichartz inequality \eqref{prop 2.1 est} turned out to be a real challenge. In fact, bilinear estimates for negatively curved surfaces require a stronger separation condition, which results in some loss when deriving linear estimates from bilinear ones.
Usually, one proves something for characteristic functions and then interpolates to extend to more general functions. 
In the proof of Theorem \ref{thm_main}, we instead consider ``funky'' characteristic functions with constant fiber length.
This requires a sharpened  form of the Hunt--Marcinkiewicz interpolation theorem, which is interesting in its own right.

\subsection{Three final remarks}

\begin{remark}
By duality and strict convexity of $L^p$ when $1<p<\infty$, the analogous result to Theorem \ref{thm_main} holds for the restriction operator. Details can be found in \cite[\S 6]{BS-parab-ext-exist}.
\end{remark}
\begin{remark}
It would be interesting to establish variants of Theorem \ref{thm_main} for perturbed and higher dimensional hyperbolic paraboloids, and similarly  for the remaining cones and hyperboloids corresponding to Cases II--III of Strichartz's paper \cite{St77}. However, a hard barrier to general $L^p$ results lies in the fact that the hyperbolic paraboloid $\Sigma\subset\R^3$ and the hyperbolic hyperboloid\footnote{Restriction estimates for the hyperbolic hyperboloid in $\R^3$ beyond Stein--Tomas  were obtained by the authors in \cite{BOSS21}. We believe that it should be possible to adapt the methods of the present paper and \cite{FS22} to obtain concentration compactness results for the hyperbolic hyperboloid as well, but the existence question is likely to be more complicated due to the lack of a scaling symmetry and the presence of antipodal points.} in $\R^3$ given by $|\boldsymbol{\xi}|^2=1+\tau^2$,  $\boldsymbol{\xi}\in\R^2$,  remain to date the unique negatively curved hypersurfaces  for which scalable extension estimates are known
beyond the Stein--Tomas range. 
\end{remark}
\begin{remark}
It is known that (except for the trivial case $p=1$) global maximizers of \eqref{eq_ExtEst} are never gaussians.\footnote{In fact, gaussians are not even critical points; see \cite{COS22}.} 
In finite fields, maximizers for the $L^2\to L^4$ extension inequality from the hyperbolic paraboloid $\{(\tau,\boldsymbol\xi)\in\mathbb F_q\times\mathbb F_q^2: \tau=\xi_1^2-\xi_2^2\}\subset \mathbb F_q^3$ have been recently classified in \cite[Theorem 1.4]{GQOS24}, but the corresponding euclidean problem remains open.
\end{remark}

\subsection{Outline}
In \S\ref{sec_symmetries}, we describe the symmetries of the extension operator $\mathcal{E}$ which will play a role in our analysis.
We perform a range localization and a frequency decomposition in \S\ref{sec_skeleton}, and then perform frequency and spatial localizations in  
 \S\ref{sec_freqloc}
and \S\ref{sec_spatialloc}, respectively. 
We then complete the proof of Theorem \ref{thm_main}.

\subsection{Terminology}
We use $X\lesssim Y$ or $Y\gtrsim X$ to denote the estimate $|X|\leq CY$ for an absolute positive constant $C$,  $X\simeq Y$ to denote the estimates $X\lesssim Y \lesssim X$, and $X\cong Y$ to denote the identity $X=CY$. 
We  often require the implied constant $C$ in the above notation to depend on additional parameters, which we will indicate by subscripts (unless explicitly omitted); thus for instance $X \lesssim_j Y$ denotes an estimate of the form $|X| \leq C_jY$ for some $C_j$ depending on $j$. 

We let $p_0$ be as in the statement of Theorem \ref{thm_main}.
Henceforth, whenever $1<p<p_0$ we write $q:=2p'$ and set $q_0:=2p_0'.$
 
\section{Symmetries}\label{sec_symmetries}

In this article, we will consider a subgroup $\scriptS_p$ of the automorphism group of $L^p(\R^2)$, which is generated by the following types of symmetries:  
\begin{enumerate}
    \item Dilations:  $D_{\boldsymbol\lambda} f(\boldsymbol\xi) = \lambda_1^{1/p}\lambda_2^{1/p}f(\lambda_1 \xi_1,\lambda_2 \xi_2)$,  ${\boldsymbol\lambda} \in (0,\infty)^2$
    \item Frequency translations:  $\tau_{\boldsymbol\xi^0}f(\boldsymbol\xi) = f(\boldsymbol\xi-\boldsymbol\xi^0)$
    \item Modulations:  $m_{(t^0,\boldsymbol x^0)}f(\xi) = e^{i(t^0,\boldsymbol x^0)\cdot(\xi_1\xi_2,\boldsymbol\xi)}f(\boldsymbol\xi)$.
\end{enumerate}
We observe that each of these symmetries corresponds to an $L^q$ automorphism via:  
\begin{gather*}
    (\scriptE D_{\boldsymbol\lambda} f)(t,\boldsymbol x) = \lambda_1^{-2/q}\lambda_2^{-2/q}(\scriptE f)([\lambda_1\lambda_2]^{-1}t,\lambda_1^{-1}x_1,\lambda_2^{-1}x_2)\\
    (\scriptE \tau_{\boldsymbol\xi^0}f)(t,\boldsymbol x) = e^{i(t,\boldsymbol x)\cdot(\xi_1^0\xi_2^0,\boldsymbol\xi^0)}(\scriptE f)(t,x_1+t\xi_2^0,x_2+t\xi_1^0)\\
    (\scriptE m_{(t^0,\boldsymbol x^0)}f)(t,\boldsymbol x) = (\scriptE f)(t+t^0,\boldsymbol x+\boldsymbol x^0).
\end{gather*}
We will distinguish one subgroup of $\scriptS_p$, namely $\scriptS_p^+$, the subgroup generated by the dilations and frequency translations.


\section{Range localization and frequency decomposition}\label{sec_skeleton}


The main result of this section is Proposition \ref{truncation}, which will allow us to identify a controllable number of axis-parallel rectangles, that we call \textit{tiles}, from which a given $f$ has significant extension. Before stating and proving this result, we establish some preliminary lemmas and propositions. Along the way, we prove a sharpened form of the Hunt--Marcinkiewicz interpolation theorem, Lemma \ref{L:general range simplification}, which may be of independent interest.

Let $\scriptD$ denote the set of dyadic tiles in $\R^2$, that is,
\begin{align*}
\scriptD = \{{\boldsymbol\zeta}+[0,2^{k_1})\times[0,2^{k_2})\colon k_1,k_2 \in \Z,~{\boldsymbol\zeta} \in 2^{k_1}\Z\times2^{k_2}\Z\}.
\end{align*}

\begin{proposition}\label{prop 2.1}
For every $p \in (1,p_0)$, there exists $\nu \in (0,1)$ such that 
\begin{align}\label{prop 2.1 est}
\|\scriptE f\|_q \lesssim \sup_{\tau \in \scriptD} \|f\chi_{\tau \cap \{|f| \leq |\tau|^{-\frac1p}\|f\|_p\}}\|_p^\nu \|f\|_p^{1-\nu},
\end{align}
for every $f \in L^p$.
\end{proposition}

Proposition \ref{prop 2.1} will follow from two lemmas below. The first one is an ``entropy reduction'' (Lemma~\ref{L:range simplification}) and the second one allows for the  isolation of a significant dyadic tile (Lemma~\ref{L:significant rectangle}).

\begin{lemma} \label{L:range simplification}
For every $p \in (1,p_0)$, there exists $\nu \in (0,1)$ such that
$$
\|\scriptE f\|_q \lesssim \sup_{j\in\Z}\sup_E \|2^j\scriptE f_{E,j}\|_q^\nu \|f\|_p^{1-\nu},
$$
for every $f \in L^p$, where $E$ ranges over measurable subsets of $\R^2$ and $f_{E,j} := \chi_{E \cap {\{2^{j-1}<|f|\leq 2^j\}}}$. 
\end{lemma}

\begin{lemma} \label{L:significant rectangle}
For every $p \in (1,p_0)$, there exists $\nu \in (0,1)$ such that
$$
\|\scriptE \chi_E\|_q \lesssim \sup_{\tau \in \scriptD:|\tau| \leq |E|}|\tau \cap E|^\frac{\nu}{p}|E|^\frac{1-\nu}{p}
$$
for every measurable $E \subseteq \R^2$.
\end{lemma}

\begin{proof}[Proof of Proposition \ref{prop 2.1}]
Fix $p \in (1,p_0)$, and let $\nu_1$ and $\nu_2$ denote the exponents $\nu$ from Lemmas \ref{L:range simplification} and \ref{L:significant rectangle}, respectively. Let $f \in L^p$ and let $E_j = \{2^{j-1} < |f| \leq 2^j\}$ for $j \in \Z$. Applying Lemmas \ref{L:range simplification} and \ref{L:significant rectangle} in sequence shows that
\begin{align*}
\|\scriptE f\|_q &\lesssim \sup_j \sup_E \sup_{\tau\colon |\tau| \leq |E\cap E_j|}2^{j\nu_1}|\tau \cap E \cap E_j|^\frac{\nu_1\nu_2}{p}|E \cap E_j|^\frac{\nu_1(1-\nu_2)}{p}\|f\|_p^{1-\nu_1}\\
&\leq \sup_j \sup_{\tau\colon |\tau| \leq |E_j|}2^{j\nu_1}|\tau \cap E_j|^\frac{\nu_1\nu_2}{p}|E_j|^\frac{\nu_1(1-\nu_2)}{p}\|f\|_p^{1-\nu_1}.
\end{align*}
Because $|f| \sim 2^j$ on $E_j$, the right-hand side above is at most
\begin{align*}
\sup_j \sup_{\tau\colon |\tau|\leq|E_j|}\|f\chi_{\tau \cap E_j}\|_p^{\nu_1\nu_2}\|f\|_p^{1-\nu_1\nu_2}.
\end{align*}
Fix $j_0$ and $\tau_0$ such that $|\tau_0| \leq |E_{j_0}|$. It suffices to show that
\begin{align*}
\|f\chi_{\tau_0\cap E_{j_0}}\|_p \lesssim \sup_\tau\|f\chi_{\tau\cap\{|f|\leq|\tau|^{-\frac{1}{p}}\|f\|_p\}}\|_p.
\end{align*}
Since $\|f\|_p^p \sim \sum_j 2^{jp}|E_j|$, it follows that
\begin{align*}
E_{j_0} \subseteq \{|f|\leq 2^{j_0}\} \subseteq \{|f|\leq C|E_{j_0}|^{-\frac{1}{p}}\|f\|_p\} \subseteq \{|f|\leq C|\tau_0|^{-\frac{1}{p}}\|f\|_p\}
\end{align*}
for some constant $C$. We can cover $\tau_0$ by at most $2C^p$ dyadic tiles $\tau_0^i$ with measure $|\tau_0^i| \leq C^{-p}|\tau_0|$. Thus
\begin{align*}
\|f\chi_{\tau_0\cap E_{j_0}}\|_p \leq \|f\chi_{\tau_0 \cap \{|f|\leq C|\tau_0|^{-\frac{1}{p}}\|f\|_p\}}\|_p &\lesssim\sup_i\|f\chi_{\tau_0^i \cap \{|f| \leq |\tau_0^i|^{-\frac{1}{p}}\|f\|_p\}}\|_p\\
&\leq\sup_\tau\|f\chi_{\tau\cap\{|f|\leq|\tau|^{-\frac{1}{p}}\|f\|_p\}}\|_p,
\end{align*}
as we needed to show.
\end{proof}

Lemma \ref{L:range simplification} will follow from a sharpened form of the Hunt--Marcinkiewicz interpolation theorem \cite{Hunt}, which we state as Lemma \ref{L:general range simplification} below. Our argument takes as inspiration a wonderful proof of Hunt--Marcinkiewicz that we learned about from lecture notes of Rowan Killip \cite{Killip}.  

Let $f$ be an arbitrary, complex-valued, measurable function.  After decomposing $f$ into its real and imaginary parts and then each of those into their positive and negative components, we expand the values of $f$ in binary; thus
\begin{equation} \label{E:binary}
f = \sum_j 2^j \chi_{F_j^1} - \sum_j 2^j\chi_{F_j^2} + i[\sum_j 2^j \chi_{F_j^3} - \sum_j 2^j \chi_{F_j^4}],
\end{equation}
with the $F^k_j$ measurable.  This expansion is uniquely determined.  To keep notation simple, we will suppress the dependence of the $F^k_j$ on $f$.  

We denote by $\|f\|_{L^{p,u}}$ the usual Lorentz quasi-norms (or, since we will always have $p>1$, their equivalent norms) and recall that for $u<\infty$,
$$
\|f\|_{L^{p,u}} \sim \bigl(\sum_{j,k}2^{ju}|F^k_j|^{\frac{u}{p}}\bigr)^{\frac1u}.
$$

\begin{lemma} \label{L:general range simplification}
Let $T$ be a sublinear operator, mapping characteristic functions of finite measure sets to measurable functions, and assume that $T$ obeys
\begin{equation} \label{E:gen rwt}
\int_F |T\chi_E(\boldsymbol x)|\, d\boldsymbol x \leq C_j |E|^{1/p_j}|F|^{1/q_j'}, \qquad j=0,1,
\end{equation}
for all bounded, measurable $E,F \subseteq \R^n$ and some $1 \leq p_j,q_j \leq \infty$, with $p_0 \neq p_1$ and $q_0 \neq q_1$.  For $0 \leq \theta \leq 1$, define $(p_\theta,q_\theta)$ in the usual way, by
$$
\tfrac1{p_\theta} := \tfrac{1-\theta}{p_0}+\tfrac\theta{p_1}, \qquad 
\tfrac1{q_\theta}:= \tfrac{1-\theta}{q_0}+\tfrac\theta{q_1}.
$$
Given a measurable $f$ with binary decomposition as in \eqref{E:binary}, for $1 \leq u < v \leq \infty$ and $0 < \theta < 1$, 
\begin{equation} \label{E:Lorentz bound}
\|Tf\|_{L^{q_\theta,v}} \lesssim \sup_{j,k} \|2^jT\chi_{F_j^k}\|_{L^{q_\theta,v}}^{1-\nu}\|f\|_{L^{p_\theta,u}}^\nu.
\end{equation}
Here $\nu$ depends on $\tfrac1u-\tfrac1v$, and 
the implicit constant depends on $C_0,C_1, \theta,u,v,$ and $p_0,p_1,q_0,q_1$.
\end{lemma}

We note that concentration-compactness methods are best-developed for $L^p$-improving estimates, namely $L^p \to L^q$ inequalities with $q > p$, due to the role of convexity.  In the proof of this lemma, we will see that ``Lorentz-improving'' can also facilitate such estimates.  

\begin{proof}[Proof of Lemma~\ref{L:general range simplification}]
To simplify notation, we will omit the subscript $\theta$ from the exponents $p_\theta,q_\theta$.  It suffices to prove that for functions $f \in L^{p,u}$, $g \in L^{q',v'}$ with $\|f\|_{L^{p,u}} = \|g\|_{L^{q',v'}}=1$,
\begin{align*}\label{E:range simpl bilinear}
    \int |Tf \, g| \lesssim \sup_{j,k}\|2^j T\chi_{F^k_j}\|_{L^{q,v}}^{1-\nu},
\end{align*}
with $\nu:=(\tfrac1u +\tfrac1{v'})^{-1} < 1$.  

By the triangle inequality and sublinearity, we may assume that $f$ and $g$ are both nonnegative.  This assumption simplifies our binary decomposition, and we can write
$$
f = \sum_j 2^j\chi_{F_j}, \qquad g = \sum_{j'}2^{j'}\chi_{G_{j'}}.
$$
Thus,
\begin{align*}
\int |Tf \, g| 
&\leq \sum_{j,j'} 2^{j+j'}\int_{G_{j'}}|T\chi_{F_j}|\\
&\leq \bigl(\sup_{j,j'} 2^{j+j'}\int_{G_{j'}}|T\chi_{F_j}|\bigr)^{1-\nu} \sum_{j,j'} \bigl(2^{j+j'}\int_{G_{j'}}|T\chi_{F_j}|\bigr)^\nu.
\end{align*}
By H\"older's inequality and $\|g\|_{L^{q',v'}} = 1$,
$$
2^{j+j'} \int_{G_{j'}}|T\chi_{F_j}| \leq \|2^jT\chi_{F_j}\|_{L^{q,v}}\|2^{j'}\chi_{G_{j'}}\|_{L^{q',v'}} \leq \|2^jT\chi_{F_j}\|_{L^{q,v}}.
$$
Thus, by \eqref{E:gen rwt}, it remains to prove that
\begin{equation} \label{E:sum nu}
\sum_{j,j'} \bigl(2^{j+j'}\min_{i=0,1}\bigl\{|F_j|^{1/p_i}|G_{j'}|^{1/q_i'}\bigr\}\bigr)^\nu \lesssim 1.  
\end{equation}
We decompose
$$
\Z = \bigcup_{a \in \Z} \scriptJ(a) = \bigcup_{a' \in \Z} \scriptJ'(a'),
$$
where 
$$
\scriptJ(a) := \{j\colon |F_j| \sim 2^a\}, \qquad \scriptJ'(a'):=\{j'\colon |G_{j'}| \sim 2^{a'}\}.
$$
It will be convenient to denote $\scriptJ(a,a') := \scriptJ(a) \times \scriptJ'(a')$.  
We also define
$$
j(a):=\sup \scriptJ(a), \qquad j'(a'):=\sup \scriptJ'(a'),
$$
with the usual convention that the empty set has $-\infty$ as its supremum.  Thus, interpreting $2^{-\infty}$ as zero, 
$$
\|f\|_{p,u} \sim \bigl(\sum_a 2^{u(j(a)+a/p)}\bigr)^{1/u}, \qquad 
\|g\|_{q',v'} \sim \bigl(\sum_{a'} 2^{v'(j'(a')+a'/q')}\bigr)^{1/v'}.
$$

We return to \eqref{E:sum nu}, whose left hand side is comparable to
\begin{align}\notag
&\sum_{a,a'}\sum_{(j,j') \in \scriptJ(a,a')} \bigl(2^{j+j'}\min_{i=0,1} \bigl\{2^{a/p_i + a'/q_i'}\bigr\}\bigr)^\nu\\ \label{E:RHS-reverse-Hunt}
&\qquad \sim \sum_{i=0,1}{\sum_{a,a'}}^{(i)}  \bigl(2^{j(a)+j'(a')}2^{a/p_i + a'/q_i'}\bigr)^\nu,
\end{align}
where the decoration $(i)$ indicates that the sum is taken over those $a,a'$ for which the minimum in the first summand is attained with index $i$.  

As our setup is completely symmetric in the indices, it suffices to bound the sum with minimizer $i=0$.  Basic arithmetic yields
\begin{align*}
&j(a)+j'(a')+a/p_0+a'/q_0' \\&\qquad
= (j(a)+a/p)+(j'(a')+a'/q') + \theta [a(1/p_0-1/p_1)+a'(1/q_0'-1/q_1')].
\end{align*}
Thus, the minimizer is $i=0$ precisely when 
$$
-k(a,a'):=a(1/p_0-1/p_1)+a'(1/q_0'-1/q_1') \leq 0.
$$
We reorganize the right hand side of \eqref{E:RHS-reverse-Hunt} based on the size of this quantity $k(a,a')$.  Namely, because $p_0 \neq p_1$ and $q_0 \neq q_1$, given $a,k \in \Z$, there are a bounded number of $a'$ with $|k(a,a')-k| < 1$, and likewise with the roles of $a,a'$ reversed.
Therefore, the right hand side of \eqref{E:RHS-reverse-Hunt} is bounded by a constant times
\begin{equation}
\begin{aligned}
&\sum_{k \geq 0} 
\sum_{(a,a'):|k-k(a,a')| < 1}2^{-\theta\nu k(a,a')}2^{\nu(j(a)+j'(a')+a/p+a'/q)}
\\
&\qquad \lesssim
\sum_{k \geq 0} 
2^{-\theta\nu k}
\bigl(\sum_a 2^{u(j(a)+a/p)}\bigr)^{\nu/u}
\bigl(\sum_{a'}
2^{v'(j'(a')+a'/q')}\bigr)^{\nu/v'}
\lesssim \|f\|_{p,u}^\nu \|g\|_{q',v'}^{\nu} \lesssim 1.
\end{aligned}
\end{equation}
Here we recall that $\nu$ was selected so that $\nu/u+\nu/v' = 1$.  We are using H\"older's inequality, since the sum over $(a,a')$ such that $|k(a,a')-k| < 1$ can be expressed as an iterated sum first over $a$, and then over a bounded number of $a'$ for each $a$, or likewise, with the roles of $a,a'$ reversed.  
\end{proof}

Next we show how Lemma \ref{L:range simplification} follows from Lemma \ref{L:general range simplification}.
\begin{proof}[Proof of Lemma \ref{L:range simplification}]
We will apply Lemma \ref{L:general range simplification} with $T = \scriptE$, $p_0 = p_0$, $q_0 = 2p_0'$, $p_1 = 1$, $q_1 = \infty$. Note that $\scriptE$ extends to a bounded linear operator from $L^{p_j}$ to $L^{q_j}$ for $j=0,1$ and therefore satisfies the restricted weak-type conditions \eqref{E:gen rwt}. Fix $p \in (1,p_0)$ and $f \in L^p$ with binary decomposition as in \eqref{E:binary}. We have $p = p_\theta$ for some $\theta \in (0,1)$. Therefore, \eqref{E:Lorentz bound} (with $u=p<q=v$) implies that
\begin{align*}
\|\scriptE f\|_q \lesssim \sup_{j,k}\|2^j\scriptE\chi_{F_j^k}\|_q^{1-\nu}\|f\|_p^\nu.
\end{align*}
Fix $j_0$ and $k_0$, and let $E_j = \{2^{j-1}<|f|\leq 2^j\}$. It suffices to show that
\begin{align*}
\|2^{j_0}\scriptE\chi_{F_{j_0}^{k_0}}\|_q \lesssim \sup_j\sup_E\|2^j\scriptE\chi_{E\cap E_j}\|_q.
\end{align*}
Observe that $F_{j_0}^{k_0} \subseteq \bigcup_{j\geq j_0} E_j$. Since the $E_j$ are pairwise disjoint, it follows that
\begin{align*}
\|2^{j_0}\scriptE\chi_{F_{j_0}^{k_0}}\|_q = \bigg\|\sum_{j\geq j_0}2^{j_0}\scriptE\chi_{F_{j_0}^{k_0}\cap E_j}\bigg\|_q &\leq \sum_{\ell \geq 0}2^{-\ell}\|2^{j_0+\ell}\scriptE\chi_{F_{j_0}^{k_0}\cap E_{j_0+\ell}}\|_q\\
&\lesssim \sup_j\sup_E\|2^j\scriptE\chi_{E\cap E_j}\|_q,
\end{align*}
as required.
\end{proof}

Next we prove Lemma \ref{L:significant rectangle} which states that, for a low-entropy set $E$, ``most'' of the norm of the extension arises from ``pieces'' that are well-adapted to ``few'' dyadic tiles.

\begin{proof}[Proof of Lemma~\ref{L:significant rectangle}]
Fix a measurable set $E \subseteq \R^2$ and decompose it (up to a set of measure zero) as $E = \bigcup_k E^k$, where
$$
E^k:=\{(\xi_1,\xi_2) \in E : 2^{k-1} < |\{\xi_2':(\xi_1,\xi_2') \in E\}| \leq 2^k\}.
$$
For each $k$ such that $|E_k| \neq 0$ (the $k$'s with $|E_k|=0$ will play no role in the argument), define $\eps_k$ to be the smallest dyadic number such that 
$$
\|\scriptE \chi_F\|_{L^q} \leq \eps_k|E^k|^{\frac1p}, \qtq{for every measurable} F \subseteq E^k,
$$
and set $\scriptK(\eps) :=\{k\colon \eps_k = \eps\}$, for dyadic $\eps>0$.  By \cite[Proposition~2.2]{BShypparab}, we may decompose $E^k$ as $E^k = \bigcup_{0 < \delta \leq \eps_k} E^k_\delta$, where
\begin{equation}\label{E:E^k_delta}
    \|\scriptE\chi_F\|_q \leq \delta|E^k|^{\frac1p}, \qtq{for every measurable} F \subseteq E^k_\delta,
\end{equation}
and $E^k_\delta \subseteq \bigcup_{\tau \in \scriptT_\delta^k} \tau$, where $\scriptT^k_\delta$ is a collection of $O(\delta^{-C})$ dyadic tiles of height (in the $\xi_2$ direction) $2^k$ and measure bounded above by $|E^k|$. 

Following the proof of  \cite[Theorem~1.1]{BShypparab} (namely, Section~3), we have by the triangle inequality that
\begin{equation}\label{E:decomp eps delta sum}
    \|\scriptE\chi_E\|_{q} \leq \sum_{0 < \eps \lesssim 1} \sum_{0 < \delta \leq \eps} \Big\|\sum_{k \in \scriptK(\eps)} \scriptE\chi_{E_\delta^k}\Big\|_{q},
\end{equation}
with sums taken over dyadic values of $\eps$ and $\delta$. 

Set $N := \lceil q \rceil$, and let $A$ be a large constant. By the triangle inequality and concavity ($q \leq N$), we see that
\begin{equation}\label{E:triang ineq multilin}
\begin{aligned}
    \Big\|\sum_{k \in \scriptK(\eps)} \scriptE \chi_{E_\delta^k}\Big\|_{q}^q &\leq \sum_{{\bf k}\in\scriptK(\eps)^N}\Big\|\prod_{i=1}^N\fE\chi_{E_\delta^{k_i}}\Big\|_{q/N}^{q/N} \lesssim \sum_{\substack{{\bf k}\in\scriptK(\eps)^N:\\k_1\leq\cdots\leq k_N}}\Big\|\prod_{i=1}^N\fE\chi_{E_\delta^{k_i}}\Big\|_{q/N}^{q/N}\\
    &\lesssim (\log\delta^{-1})^{N+1} \sum_{k \in \scriptK(\eps)}\|\scriptE\chi_{E_\delta^k}\|_{q}^q + \sum_{\substack{{\bf k}\in\scriptK(\eps)^N:\\k_1\leq\cdots\leq k_N,\\|k_1-k_N|\geq A\log\delta^{-1}}}\Big\|\prod_{i=1}^N \scriptE \chi_{E_\delta^{k_i}}\Big\|_{q/N}^{q/N}.
    \end{aligned}
\end{equation}
If $A$ is sufficiently large, then by \cite[Lemma~3.2]{BShypparab},
$$
\|\scriptE\chi_{E_\delta^{k_1}} \scriptE\chi_{E_\delta^{k_N}}\|_{q/2} \lesssim 2^{-c_0|k_1-k_N|}\max\{|E^{k_1}|^{\frac2p}, |E^{k_N}|^{\frac2p}\}.
$$
Therefore by H\"older's inequality and elementary combinatorics,
\begin{align*}
    &\sum_{\substack{{\bf k}\in\scriptK(\eps)^N:\\k_1\leq\cdots\leq k_N,\\|k_1-k_N|\geq A\log\delta^{-1}}}\Big\|\prod_{i=1}^N \scriptE \chi_{E_\delta^{k_i}}\Big\|_{q/N}^{q/N} \\
    &\qquad \lesssim \sup_{k\in\scriptK(\eps)}\|\scriptE\chi_{E_\delta^k}\|_q^{q-p}\sum_{\substack{{\bf k}\in\scriptK(\eps)^N:\\k_1\leq\cdots\leq k_N,\\|k_1-k_N|\geq A\log\delta^{-1}}} 2^{-\frac{c_0p}{N}|k_1-k_N|}\max_i|E^{k_i}|\\
    &\qquad \lesssim \sup_{k\in\scriptK(\eps)}\|\scriptE\chi_{E_\delta^k}\|_q^{q-p} \sum_{k \in \scriptK(\eps)} \sum_{l\geq0} 2^{-\frac{c_0p}{N}l}l^N\max_{k':|k-k'|\leq l}|E^{k'}|\\
    &\qquad \lesssim
    \sup_{k\in\scriptK(\eps)}\|\scriptE\chi_{E_\delta^k}\|_q^{q-p}\sum_{l\geq0}2^{-\frac{c_0p}{N}l}l^{N+1}\sum_{k\in\scriptK(\eps)}|E^k|\\
    &\qquad \lesssim
    \sup_{k\in\scriptK(\eps)}\|\scriptE\chi_{E_\delta^k}\|_q^{q-p}|E|. 
\end{align*}
Returning to \eqref{E:triang ineq multilin},
\begin{equation}\label{E:k eps delta sum}
    \Big\|\sum_{k \in \scriptK(\eps)}\scriptE\chi_{E_\delta^k}\Big\|_q^q \lesssim (\log \delta^{-1})^{N+1} \sup_{k \in \scriptK(\eps)}\|\scriptE \chi_{E_\delta^k}\|_q^{q-p}|E|.
\end{equation}
By the triangle inequality and \eqref{E:E^k_delta},
\begin{align*}
    \|\scriptE \chi_{E_\delta^k}\|_q &\lesssim \|\scriptE \chi_{E_\delta^k}\|_q^{\frac{C+1}{C+2}}\Big(\sum_{\tau \in \scriptT_\delta^k}|E_\delta^k \cap \tau|^\frac{1}{p}\Big)^\frac{1}{C+2}\\
    &\lesssim \delta^{\frac{C+1}{C+2}} |E^k|^\frac{C+1}{(C+2)p} \delta^{-\frac{C}{C+2}}\sup_{\tau \in \scriptT^k_\delta} |E_\delta^k \cap \tau|^\frac{1}{(C+2)p}\\
    &\leq \delta^{\frac1{C+2}}|E|^\frac{C+1}{(C+2)p} \sup_{\tau \in \scriptT_\delta^k}|E\cap\tau|^\frac{1}{(C+2)p}.
\end{align*}
For $\tau \in \scriptT_\delta^k$, we have $|\tau| \leq |E^k| \leq |E|$.  Thus inserting the above bound on $\|\scriptE\chi_{E_\delta^k}\|_q$ into \eqref{E:decomp eps delta sum}, we get
$$
\|\scriptE\chi_E\|_q \lesssim \sum_{0 < \eps \lesssim 1} \sum_{0 < \delta \leq \eps} \delta^\theta \sup_{|\tau| \leq |E|} |E\cap\tau|^\frac{\nu}{p}|E|^\frac{1-\nu}{p},
$$
for sufficiently small $\theta > 0$ and $\nu :=\frac{1}{C+2}(1-\frac{p}{q})$.  The lemma follows by summing the geometric series in $\delta$ and $\varepsilon$.
\end{proof}

Now we are ready to state and prove the main result of this section.

\begin{proposition}\label{truncation}
There exists a sequence $\rho_j \searrow 0$ such that for every $f \in L^p$ with $\|f\|_p = 1$, there exists a sequence $\tau^j$ of dyadic tiles such that if
\begin{align*}
    r^0 := f, \qquad    g^j := r^{j-1}\chi_{\tau^j \cap \{|f|\leq|\tau^j|^{-1/p}\}},\qquad 
    r^j := r^{j-1}-g^j,
\end{align*}
then $\|\scriptE h^j\|_q \leq \rho_j$ for all $j$ and $h^j := r^j\chi$ with $\chi$ any measurable characteristic function.
\end{proposition}

\begin{proof}
We claim that, given $r^{j-1}$, we can select $\tau^j \in \scriptD$ so that
\begin{align}\label{maximal property}
    \|g^j\|_p = \max_{\tau \in \scriptD}\|r^{j-1}\chi_{\tau \cap \{|f| \leq |\tau|^{-1/p}\}}\|_p,
\end{align}
in the notation of the proposition. Indeed, routine applications of the dominated convergence theorem show that $\|r^{j-1}\chi_{\tau \cap \{|f| \leq |\tau|^{-1/p}\}}\|_p \rightarrow 0$ when $|\tau| \rightarrow 0$ or $|\tau|\rightarrow \infty$ or $\dist(\tau,0)\rightarrow\infty$, while a simple truncation argument also implies convergence to zero when $|\tau| \sim 1$ and $\diam(\tau) \rightarrow \infty$.  Since the set of tiles $\tau$ satisfying $|\tau|\sim 1$, $\dist(\tau,0) \lesssim 1$, and $\diam(\tau)\lesssim 1$ is finite,  the claim follows.  

Using this construction (namely, maximality of $\tau^j$), $\|g^j\|_p$ is decreasing in $j$.  Moreover, by disjointness of the supports of the $g_j$, we have $\sum_j \|g_j\|_p^p \leq \|f\|_p^p = 1$.  Therefore $\|g_j\|_p \leq j^{-1/p}$.  Now, let $h^j := r^{j}\chi$.  From Proposition \ref{prop 2.1} we have
\begin{align*}
    \|\scriptE h^j\|_q &\lesssim \sup_{\tau \in \scriptD}\|h^j\chi_{\tau \cap \{|h^j|\leq|\tau|^{-1/p}\|h^j\|_p\}}\|_p^\nu\\
    &\leq \sup_{\tau \in \scriptD}\|r^j\chi_{\tau \cap \{|f|\leq|\tau|^{-1/p}\}}\|_p^\nu = \|g^{j+1}\|_p^\nu \leq j^{-\frac{\nu}{p}},
\end{align*}
where we are using that $h^j=r^j=f$ wherever $h^j \neq 0$ and that $\|h^j\|_p \leq \|f\|_p = 1$.  Thus, the proposition holds with $\rho_j := Cj^{-\nu/p}$ for some constant $C$.
\end{proof}

\section{Frequency localization}\label{sec_freqloc}
The main result of this section is Proposition \ref{freq loc}, which states that a near-maximizer is well adapted to a single tile. 

\begin{proposition}\label{freq loc}
There exist functions $\delta\colon (0,1] \to (0,\frac{1}{2}{\bf A}_p]$,  $R\colon (0,1] \to [1,\infty)$, and $S\colon L^p(\R^2) \to \scriptS_p^+$ such that for every $\eps \in (0,1]$ and $f \in L^p$, if $\|\scriptE f\|_q \geq ({\bf A}_p-\delta_\eps)\|f\|_p$, then $\|(S_f)f\|_{L^p(\{|\boldsymbol\xi|\leq R_\eps,|(S_f)f|\leq R_\eps\|f\|_p\})}\geq (1-\eps)\|f\|_p$.
\end{proposition}

\begin{proof}
We begin by defining the mapping $S\colon L^p \rightarrow \mathcal{S}_p^+$. Fix $f \in L^p$ with $\|f\|_p = 1$, and recall the inductive procedure introduced in Proposition \ref{truncation}. There exists a symmetry $S_f \in \mathcal{S}_p^+$ such that if we replace $f$ by $(S_f)f$ in that procedure, then the tile $\tau^1$ becomes the unit square. We set $S(f) := S_f$. For unnormalized $0\not \equiv f \in L^p$, we define $S(f):= S(\frac{f}{\|f\|_p})$, and we set $S(0)$ to equal the identity.  

The proposition is now implied by the statement that for every $\varepsilon \in (0,1]$, there exist $\delta \in (0,\frac{1}{2}{\bf A}_p]$ and $R \in [1,\infty)$ such that if $f \in L^p$ and $\|\fE f\|_q \geq ({\bf A}_p-\delta)\|f\|_p$ then $$\|(S_f)f\|_{L^p(\{|\boldsymbol\xi|\leq R,|(S_f)f|\leq R\|f\|_p\})}\geq (1-\eps)\|f\|_p.$$ Suppose this statement is false, for contradiction. Then there exists an $\varepsilon > 0$ such that for any choice of $\delta$ and $R$, there exists $f \in L^p$ with $\|f\|_p = 1$ and $\|\fE f\|_q \geq {\bf A}_p - \delta$ but
\begin{align*}
\|(S_f)f\|_{L^p(\{|\boldsymbol\xi|\geq R\} \cup \{|(S_f)f|\geq R\})} \geq \varepsilon.
\end{align*}
In particular, taking $\delta := 1/n$ and $R := n$ for each $n \in \N$, we get a sequence of functions $\{f_n\} \subset L^p$ with $\|f_n\|_p = 1$ and $\|\fE f_n\|_q \rightarrow {\bf A}_p$ but
\begin{align}\label{contradiction assumption}
\|S_nf_n\|_{L^p(\{|\boldsymbol\xi|\geq n\} \cup \{|S_nf_n|\geq n\})} \geq \varepsilon,
\end{align}
where $S_n := S_{f_n}$ (as defined above).

Now, let $\rho$ be a small positive number depending only on $\varepsilon$, to be determined momentarily.  By Proposition \ref{truncation}, there exists $J \in \N$ such that for every $f \in L^p$ with $\|f\|_p = 1$, there exist dyadic tiles $\tau^j$, $1 \leq j \leq J$, with the following property: If we inductively define
\begin{align}\label{inductive construction}
   r^0 := f, \quad\quad g^j := r^{j-1}\chi_{\tau^j \cap \{|f|\leq|\tau^j|^{-\frac{1}{p}}\}}, \quad\quad r^j := r^{j-1}-g^j, \quad\quad 1 \leq j \leq J, 
\end{align}
then for all functions $h^J$ of the form $h^J = r^J\chi$ with $\chi$ a measurable characteristic function, we have
\begin{align*}
\|\fE h^J\|_q < \rho.
\end{align*}
Applying the construction \eqref{inductive construction} to $S_nf_n$ gives the following:  For each $n \in \N$, there exist dyadic tiles $\tau_n^j$, $1 \leq j \leq J$, such that if
\begin{align*}
r_n^0 := S_nf_n, \quad\quad g_n^j := r_n^{j-1}\chi_{\tau_n^j \cap \{|S_nf_n|\leq|\tau_n^j|^{-\frac{1}{p}}\}}, \quad\quad r_n^j := r_n^{j-1}-g_n^j, \quad\quad 1 \leq j \leq J, 
\end{align*}
then
\begin{align}\label{remainder estimate}
\|\fE h_n^J\|_q < \rho
\end{align}
for all functions $h_n^J$ of the form $h_n^J = r_n^J\chi$ with $\chi$ a measurable characteristic function.

Let $F_n := S_nf_n - r_n^J$.  By \eqref{remainder estimate}, our hypotheses on $\{f_n\}$, and the disjointness of the supports of $F_n$ and $r_n^J$, we have
\begin{align*}
{\bf A}_p-\rho \leq \liminf_n\|\fE F_n\|_q \leq \liminf_n {\bf A}_p\|F_n\|_p &= \liminf_n {\bf A}_p(1-\|r_n^J\|_p^p)^\frac{1}{p}\\ &\leq {\bf A}_p - C_p\limsup_n\|r_n^J\|_p^p
\end{align*}
for some constant $C_p$. Thus, after passing to a subsequence, we may assume that $\|r_n^J\|_p \lesssim \rho^\frac{1}{p}$, which for $\rho$ sufficiently small implies that
\begin{align}\label{S_nf_n-F_n}
\|S_nf_n-F_n\|_p < \frac{\varepsilon}{2}.
\end{align}

By the proof of Proposition \ref{truncation} (see \eqref{maximal property}), Proposition \ref{prop 2.1}, and the fact that $\|\fE f_n\|_q \rightarrow {\bf A}_p$, we have
\begin{align*}
\|g_n^1\|_p = \max_{\tau \in \scriptD}\|S_nf_n\chi_{\tau\cap\{|S_nf_n|\leq|\tau|^{-\frac{1}{p}}\}}\|_p \gtrsim \|\fE f_n\|_q^\frac{1}{\nu} \gtrsim 1
\end{align*}
for $n$ sufficiently large.  By our choice of the symmetry function $S$, we have that $\tau_n^1 = [0,1)^2$ for all $n$.  The remaining tiles may be written as
\begin{align*}
\tau_n^j = \boldsymbol\zeta_n^j + [0,2^{k_n^j})\times[0,2^{l_n^j}),
\end{align*}
where $k_n^j,l_n^j \in \Z$ and $\boldsymbol\zeta_n^j \in 2^{k_n^j}\Z \times 2^{l_n^j}\Z$.

Passing to a subsequence, we may assume that for each $1 \leq j \leq J$, either $k_n^j$ is constant (in $n$) or $|k_n^j|\rightarrow \infty$ and that either $l_n^j$ is constant or $|l_n^j| \rightarrow \infty$ and that either $\boldsymbol\zeta_n^j$ converges (as $n \to \infty$)  or $|\boldsymbol\zeta_n^j|\rightarrow \infty$. We say that an index $j$ is {\it good} if $k_n^j$ and $l_n^j$ are constant in $n$ and $\boldsymbol\zeta_n^j$ converges.  In fact, if $j$ is good, $\boldsymbol\zeta_n^j$ is eventually constant, so, passing to a subsequence, we may assume that for good $j$, $\boldsymbol\zeta_n^j$ is constant in $n$.  We say that $j$ is {\it bad} if it is not good. Note that the definition of $S_n$ implies that $j=1$ is good. 

Decompose $F_n$ as $F_n = G_n+B_n$, where
\begin{align*}
G_n := \sum_{j\text{ good}} g_n^j, \quad\quad B_n := F_n-G_n.
\end{align*}
The definition of the $g_n^j$ implies the existence of a bounded set $K$ and a finite constant $C$ such that $\supp G_n \subseteq K$ for sufficiently large $n$ and $\|S_nf_n\|_{L^\infty(K)} \leq C$ for all $n$.  Thus, by \eqref{S_nf_n-F_n} and \eqref{contradiction assumption}, we have
\begin{align*}
\liminf_n\|B_n\|_p &\geq \liminf_n\|S_nf_n-G_n\|_{L^p(\{|\xi|\geq n\}\cup\{|S_nf_n|\geq n\})} - \frac{\varepsilon}{2}\\
&=\liminf_n\|S_nf_n\|_{L^p(\{|\xi|\geq n\}\cup\{|S_nf_n|\geq n\})} - \frac{\varepsilon}{2}\\ &\geq \frac{\varepsilon}{2}.
\end{align*}
This implies that, after passing to a subsequence, 
\begin{align}\label{good part bound}
\|G_n\|_p \leq (1-(\varepsilon/2)^p)^\frac{1}{p} \leq 1-c\varepsilon^p
\end{align}
for some $c > 0$.  Since $1 \lesssim \|g_n^1\|_p \leq \|G_n\|_p$, we consequently have
\begin{align}\label{bad part bound}
\|B_n\|_p = (\|F_n\|_p^p - \|G_n\|_p^p)^\frac{1}{p} \leq 1-c_0
\end{align}
for some $c_0 \gtrsim 1$.  We will use the following lemma.  

\begin{lemma}\label{decoupling}
After passing to a subsequence,
\begin{align*}
\lim_{n\rightarrow \infty} (\|\fE F_n\|_q^q -\|\fE G_n\|_q^q - \|\fE B_n\|_q^q) = 0.
\end{align*}
\end{lemma}
Assuming Lemma~\ref{decoupling} (which will be proved shortly) for now, we continue with the proof of Proposition \ref{freq loc}.  Using \eqref{good part bound}, \eqref{bad part bound}, and the fact that $q > p$, we have that
\begin{align*}
{\bf A}_p - \rho \leq \liminf_n \|\fE F_n\|_q &= \liminf_n(\|\fE G_n\|_q^q + \|\fE B_n\|_q^q)^{1/q}\\
&\leq {\bf A}_p\liminf_n(\|G_n\|_p^q + \|B_n\|_p^q)^{1/q}\\
&\leq {\bf A}_p \liminf_n \max\{\|G_n\|_p,\|B_n\|_p\}^{1-p/q}\|F_n\|_p^{p/q}\\
&\leq {\bf A}_p\max\{1-c\varepsilon^p,1-c_0\}^{1-p/q}.
\end{align*}
Choosing $\rho$ sufficiently small gives a contradiction.
\end{proof}

\begin{proof}[Proof of Lemma \ref{decoupling}]
By the elementary inequality
\begin{align*}
||a+b|^q-|a|^q-|b|^q| \leq C_q(|a||b|^{q-1}+|a|^{q-1}|b|), \quad\quad a,b \in \C, \quad q \geq 1,
\end{align*}
H\"older's inequality, and the fact that $\|\fE g_n^j\|_q \lesssim \|g_n^j\|_p \leq 1$, we have
\begin{align*}
|\|\fE F_n\|_q^q - \|\fE G_n\|_q^q - \|\fE B_n\|_q^q| &\lesssim_{q,J} \sup_{j\text{ good, } j' \text{ bad}}\int(|\fE g_n^j|^{q-1}|\fE g_n^{j'}|+|\fE g_n^j||\fE g_n^{j'}|^{q-1})\\
&\lesssim \sup_{j\text{ good, } j' \text{ bad}}\bigl(\|\fE g_n^j\|_q^{q-2} + \|\fE g_n^{j'}\|_q^{q-2}\bigr)\|\fE g_n^j \fE g_n^{j'}\|_{\frac q2}\\
&\lesssim \sup_{j\text{ good, } j' \text{ bad}} \|\fE g_n^j \fE g_n^{j'}\|_{\frac q2}.
\end{align*}

Since there are only finitely many indices $1 \leq j,j' \leq J$, it suffices to prove that for each good $j$ and bad $j'$, each subsequence has a further subsequence along which 
$$
\lim_{n \to \infty} \|\scriptE g_n^j \scriptE g_n^{j'}\|_{q/2} = 0.
$$

Fix $j$ good, $j'$ bad, and a subsequence.  To simplify notation, we set $g_n:=g_n^j$ and $b_n:=g_n^{j'}$.  After rescaling and translating, we may assume that $\tau_n^j = Q := [0,1)^2$, so we set $\tau_n := \tau_n^{j'}$, $k_n:=k_n^{j'}$, $l_n:=l_n^{j'}$, $\boldsymbol\zeta_n:=\boldsymbol\zeta_n^{j'}$.  After passing to a further subsequence, one of the following holds:  
\begin{enumerate}
    \item[(1)]  $|k_n+l_n| \rightarrow \infty$,
    \item[(2)] $k_n+l_n$ is constant and $|k_n|+|l_n|\rightarrow \infty$,
    \item[(3)] $k_n$ and $l_n$ are constant and $|\boldsymbol\zeta_n| \rightarrow \infty$.
\end{enumerate}

\emph{Case 1.} Assume that $|k_n+l_n| \rightarrow \infty$. We recall that $|\tau_n| = 2^{k_n+l_n}$. 
  Passing to a  subsequence, either $k_n+l_n \rightarrow \infty$ or $k_n+l_n\rightarrow -\infty$.  We take $q_0<q_+<q<q_-$ with $\frac{2}{q}=\frac{1}{q_-}+\frac{1}{q_+}$, and we set $p_\pm := (\frac{q_\pm}{2})'$. Note that $p_- < p < p_+ < p_0$ and that $\|g_n\|_{p_\pm} \lesssim 1$.  If $k_n+l_n\rightarrow\infty$, then
\begin{align*}
\|\fE g_n \fE b_n\|_\frac{q}{2} \leq \|\fE g_n\|_{q_-}\|\fE b_n\|_{q_+} \lesssim \|g_n\|_{p_-}\|b_n\|_{p_+} \lesssim |\tau_n|^{\frac{1}{p_+}-\frac{1}{p}} = o(1). 
\end{align*}
Otherwise, $k_n+l_n\rightarrow-\infty$ and
\begin{align*}
\|\fE g_n \fE b_n\|_\frac{q}{2} \leq \|\fE g_n\|_{q_+}\|\fE b_n\|_{q_-} \lesssim \|g_n\|_{p_+}\|b_n\|_{p_-} \lesssim |\tau_n|^{\frac{1}{p_-}-\frac{1}{p}} = o(1).
\end{align*}

\emph{Case 2.} Assume that $k_n+l_n \equiv M$ and $|k_n|+|l_n|\rightarrow\infty$. Passing to a subsequence, we may assume by symmetry that $k_n \rightarrow \infty$ and $l_n \rightarrow -\infty$.  Recall that $\boldsymbol\zeta_n = (\zeta_{n,1},\zeta_{n,2})$ denotes the lower left corner of $\tau_n$.
Decompose $Q$ and $\tau_n$ as
\begin{align*}
Q = \bigcup_{i=0}^\infty Q^{i}_n, \quad\quad Q^{i}_n := \begin{cases}
Q \cap \{|\xi_2-\zeta_{n,2}| \lesssim 2^{l_n}\}, &i=0\\
Q \cap \{|\xi_2-\zeta_{n,2}| \sim 2^{l_n+i}\}, &i>0
\end{cases}
\end{align*}
and
\begin{align*}
\tau_n = \bigcup_{i=0}^\infty \tau_n^{i}, \quad\quad \tau_n^{i} :=
\begin{cases}
\tau_n \cap \{|\xi_1| \lesssim 1\}, &i=0\\
\tau_n \cap \{|\xi_1| \sim 2^{i}\}, &i>0,
\end{cases}
\end{align*}
and let $g_n^{i} := g_n\chi_{Q_n^{i}}$ and $b_n^i := b_n\chi_{\tau_n^{i}}$.  By our construction and H\"older's inequality,
\begin{equation}\label{E:Egnji}
    \|\fE g_n^{i}\|_q \lesssim \|g_n^{i}\|_p \lesssim \|g_n\|_\infty |Q_n^{i}|^{1/p} \lesssim \min\{1,2^{(l_n+i)/p}\},
\end{equation}
and similarly,
\begin{equation} \label{E:Egnj'i}
\|\fE b_n^i\|_q \lesssim \|b_n\|_\infty |\tau_n^{i}|^{1/p} \lesssim |\tau_n|^{-1/p}|\tau_n^{i}|^{1/p} \lesssim \min\{1,2^{(i-k_n)/p}\}.
\end{equation}
Furthermore, $Q_n^{i}$ is nonempty for at most 2 values of $i$ with $i \geq -l_n$, and similarly, $\tau_n^{i}$ is nonempty for at most 2 values of $i$ with $i \geq k_n$.   

By the triangle inequality, we have
\begin{align*}
\|\fE g_n \fE b_n\|_\frac{q}{2} \leq \sum_{i=0}^\infty\sum_{i'=0}^\infty\|\fE g_n^{i} \fE b_n^{i'}\|_\frac{q}{2}.
\end{align*}
By Cauchy--Schwarz, \eqref{E:Egnji},  \eqref{E:Egnj'i}, and the observation immediately following, the sum on those terms with $i \lesssim 1$ or $i' \lesssim 1$ is bounded above by
\begin{align*}
\Big(\sum_{i\lesssim 1}\sum_{i'=0}^\infty + \sum_{i=0}^\infty\sum_{i'\lesssim 1}\Big)\|\fE g_n^{i} \fE b_n^{i'}\|_\frac{q}{2} &\leq \Big(\sum_{i\lesssim 1}\sum_{i'=0}^\infty + \sum_{i=0}^\infty\sum_{i'\lesssim 1}\Big) \|\fE g_n^{i}\|_q\| \fE b_n^{i'}\|_q\\ 
&\lesssim 2^{l_n/p}+2^{-k_n/p} \to 0.
\end{align*}

We now sum the remaining terms, for which $i,i'\geq C$, for some large constant $C$.  Each $Q_n^{i}$ is contained in a bounded number of tiles of width $1$ and height $2^{\min\{l_n+i,0\}}$, and is also contained in the unit cube, and each $\tau_n^{i'}$ is contained in a bounded number of tiles of width $2^{\min\{i',k_n\}}$ and height $2^{l_n}$.  Since $i\geq C$ and $i'\geq C$, the tiles that make up $Q_n^i$ are separated from those that make up $\tau_n^{i'}$ by a distance of $2^{i'}$ horizontally and $2^{l_n+i}$ vertically.  Thus, by the bilinear extension theory for the hyperbolic paraboloid \cite{Le06, Va05}, there exists $r(q) < (\frac{q}{2})'$ such that for all $r(q) < r < (\frac{q}{2})'$,
\begin{align*}
\|\fE g_n^i \fE b_n^{i'}\|_\frac{q}{2} &\lesssim 2^{(i'+l_n+i)(2-\frac{4}{q}-\frac{2}{r})}\|g_n^i\|_r\|b_n^{i'}\|_r\\
&\lesssim 2^{(i'+l_n+i)(2-\frac{4}{q}-\frac{2}{r})}|Q_n^i|^\frac{1}{r}|\tau_n^{i'}|^\frac{1}{r}|\tau_n|^{-\frac{1}{p}}.
\end{align*}
Choosing $r$ sufficiently close to $(q/2)'$ and noting that $|Q_n^i|\leq\min\{1,2^{l_n+i}\}$,  $|\tau_n^{i'}|\leq\min\{2^M, 2^{l_n+i'}\}$, and that there are at most two nonempty $Q_n^i$ (respectively $\tau_n^{i'}$) with $i\geq -l_n$ (respectively $i'\geq k_n$), some arithmetic shows that
\begin{align*}
\sum_{i=C}^\infty\sum_{i'=C}^\infty\|\fE g_n^i \fE b_n^{i'}\|_\frac{q}{2} \lesssim 2^{M(1-\frac{1}{p})+k_n(2-\frac{4}{q}-\frac{2}{r})} = o(1).
\end{align*}

\emph{Case 3.} Assume that $k_n \equiv k$, $l_n \equiv l$, and $|\boldsymbol\zeta_n|\rightarrow \infty$.  Passing to a subsequence, we may assume by symmetry that $|\zeta_{n,1}|\rightarrow \infty$, where $\boldsymbol\zeta_n =: (\zeta_{n,1},\zeta_{n,2})$. Let $K$ be an integer such that $3\cdot2^K \leq \dist_1(Q,\tau_n) \leq 6\cdot2^K$, where $\dist_1(A,B)$ denotes the distance between $A$ and $B$ with respect to the $\xi_1$ coordinate direction.  For $n$ sufficiently large, we have $K \geq \max\{0,k\}$. Let $L := \max\{0,l\}$.  Then $Q$ is contained in the $2^K \times 2^L$ dyadic tile $\tilde{Q} := [0,2^K)\times[0,2^L)$.  We also let $\tilde{\tau}$ denote the $2^K \times 2^L$ dyadic tile containing $\tau_n$. (We suppress the dependence on $n$ to simplify later notation.)  The tile $\tilde{Q}$ and $\tilde{\tau}$ are horizontally separated by a distance of at least $2^K$.

For each integer $m \leq L$, decompose $\tilde{Q}$ and $\tilde{\tau}$ as
\begin{align*}
\tilde{Q} = \bigcup_i \tilde{Q}_{m,i}, \quad\quad \tilde{\tau} = \bigcup_i \tilde{\tau}_{m,i}
\end{align*}
where $\tilde{Q}_{m,i}$ (respectively $\tilde{\tau}_{m,i}$) are $2^{K}\times 2^m$ dyadic tiles contained in $\tilde{Q}$ (respectively $\tilde{\tau}$). Let $\boldsymbol a\in 2^K\Z \times 2^L\Z$ denote the lower left corner of $\tilde{\tau}$, i.e.~ $\tilde{\tau} =: \boldsymbol a+[0,2^K)\times[0,2^L)$.  We write $(m,i)\sim(m,i')$, and say that $\tilde{Q}_{m,i}$ and $\tilde{\tau}_{m,i'}$ are related, if the dyadic tiles $\tilde{Q}_{m,i}$ and $\tilde{\tau}_{m,i'}-\boldsymbol a$ are non-adjacent but have adjacent dyadic parents. If $(m,i)\sim(m,i')$, then $\tilde{Q}_{m,i}$ and $\tilde{\tau}_{m,i'}$ are separated by a distance of (at least) $2^K$ horizontally and $2^m$ vertically. Up to a set of measure zero, we have
\begin{align*}
\tilde{Q}\times\tilde{\tau} = \bigcup_{m\leq L}\bigcup_{i,i': (m,i)\sim(m,i')}\tilde{Q}_{m,i}\times\tilde{\tau}_{m,i'}.
\end{align*}
Let $g_{m,i} := g_n\chi_{\tilde{Q}_{m,i}}$ and $b_{m,i} := b_n\chi_{\tilde{\tau}_{m,i}}$ (again, suppressing dependence on $n$), and fix $r(q) < r < (\frac{q}{2})'$ as in Case 2 above.  Then following the Tao--Vargas--Vega \cite{TVV98} bilinear-to-linear argument,
\begin{align*}
\|\fE g_n \fE b_n\|_\frac{q}{2} &\leq \sum_{m\leq L}\bigg\|\sum_{i,i' : (m,i)\sim(m,i')}\fE g_{m,i}\fE b_{m,i'}\bigg\|_\frac{q}{2}\\
&\lesssim\sum_{m\leq L}\bigg(\sum_{i,i':(m,i)\sim(m,i')}\|\fE g_{m,i}\fE b_{m,i'}\|_\frac{q}{2}^t\bigg)^\frac{1}{t}, \quad\quad t:=\min\Big\{\frac{q}{2},\Big(\frac{q}{2}\Big)'\Big\}\\
&\lesssim\sum_{m\leq L}2^{(K+m)(2-\frac{4}{q}-\frac{2}{r})}\bigg(\sum_{i,i':(m,i)\sim(m,i')}\|g_{m,i}\|_r^t\|b_{m,i'}\|_r^t\bigg)^\frac{1}{t}\\
&\lesssim\sum_{m\leq L}2^{(K+m)(2-\frac{4}{q}-\frac{2}{r})}\bigg(\sum_{i,i':(m,i)\sim(m,i')}|\tilde{Q}_{m,i} \cap Q|^\frac{t}{r}|\tilde{\tau}_{m,i'}\cap \tau_n|^\frac{t}{r}|\tau_n|^{-\frac{t}{p}}\bigg)^\frac{1}{t}\\
&\lesssim \sum_{m\leq L}2^{(K+m)(2-\frac{4}{q}-\frac{2}{r})}2^\frac{L-m}{t}2^\frac{m}{r}2^\frac{k+m}{r}2^{-\frac{k+l}{p}}\\
&\sim 2^{K(2-\frac{4}{q}-\frac{2}{r}) + L(2-\frac{4}{q})-\frac{k+l}{p}+\frac{k}{r}}\\
&= o(1).
\end{align*}
This completes the proof of Lemma \ref{decoupling}.
\end{proof}

\section{Spatial localization}\label{sec_spatialloc}

In this section, we complete the proof of  Theorem \ref{thm_main}. This final part of the argument closely follows the methods presented in Sections 4 and 5 of \cite{BS-parab-ext-exist}, so we will omit some details, such as the proof of Proposition \ref{L^p profile decomposition} below.

We begin with the $L^2$-theory, which was developed by Rogers--Vargas \cite{RogersVargas} and Dodson--Marzuola--Pausader--Spirn \cite{DMPS18}.  Specifically,  \cite[Theorem~12]{DMPS18} implies the following:

\begin{theorem}\label{L^2 profile decomposition}
Let $\{f_n\}$ be a bounded sequence in $L^2$. After passing to a subsequence, there exists $J_0 \in \N \cup \{\infty\}$, symmetries $S_n^j \in \scriptS_2$, nonzero profiles $\phi^j \in L^2$, and errors $r_n^J \in L^2$ such that for each $J < J_0$ we have
\begin{align*}
f_n = \sum_{j=1}^J S_n^j\phi^j + r_n^J,
\end{align*}
and
\begin{enumerate}
\item[(i)] For all $j \neq j'$, $(S_n^j)^{-1}S_n^{j'} \rightharpoonup 0$ in the weak operator topology;
\item[(ii)] For all $J<J_0$, $\lim_{n\rightarrow\infty}(\|f_n\|_2^2-\sum_{j=1}^J\|\phi^j\|_2^2-\|r_n^J\|_2^2) = 0$;
\item[(iii)] For all $J<J_0$, $\lim_{n\rightarrow\infty}(\|\fE f_n\|_4^4-\sum_{j=1}^J\|\fE\phi^j\|_4^4-\|\fE r_n^J\|_4^4) = 0$;
\item[(iv)] For all $j$, $\phi^j = \operatorname{wk-lim}(S_n^j)^{-1}f_n$;
\item[(v)] The extensions of the errors tend to zero: $\lim_{J\rightarrow J_0}\limsup_{n\rightarrow\infty}\|\fE r_n^J\|_4 = 0$.
\end{enumerate}
\end{theorem}

As in \cite{BS-parab-ext-exist}, we can upgrade Theorem \ref{L^2 profile decomposition} to an $L^p$-based profile decomposition for frequency localized functions (see \cite[Proposition 4.1]{BS-parab-ext-exist} for details).

\begin{proposition}\label{L^p profile decomposition}
Let $R>0$ and let $\{f_n\}$ be a sequence of measurable functions supported on $\{|\boldsymbol \xi|\leq R\}$ and satisfying $|f_n| \leq R$. After passing to a subsequence, there exist $J_0 \in \N \cup \{\infty\}$, $(t_n^j,\boldsymbol x_n^j) \in \R^{1+2}$, bounded measurable functions $\phi^j$ supported on $\{|\boldsymbol\xi|\leq R\}$, and remainders $r_n^J$ such that for each $J < J_0$ we have
\begin{align*}
f_n = \sum_{j=1}^Je^{i(t_n^j,\boldsymbol x_n^j)(\xi_1\xi_2,\boldsymbol\xi)}\phi^j + r_n^J,
\end{align*}
and
\begin{enumerate}
\item[(i)]For all $j\neq j'$, $\lim_{n\rightarrow\infty}(|t_n^j-t_n^{j'}| + |\boldsymbol x_n^j-\boldsymbol x_n^{j'}|) = \infty$;
\item[(ii)] If $\tilde{p} := \max\{p,p'\}$, then $\liminf_{n\rightarrow\infty}(\|f_n\|_p - (\sum_{j=1}^{J_0}\|\phi^j\|_p^{\tilde{p}})^\frac{1}{\tilde{p}} \geq 0$;
\item[(iii)] For all $J < J_0$, $\lim_{n\rightarrow\infty}(\|\fE f_n\|_q^q-\sum_{j=1}^J\|\fE\phi^j\|_q^q-\|\fE r_n^J\|_q^q) = 0$;
\item[(iv)] For all $j$, $\phi^j = \operatorname{wk-lim}e^{-i(t_n^j,\boldsymbol x_n^j)(\xi_1\xi_2,\boldsymbol \xi)}f_n$;
\item[(v)] The extensions of the errors tend to zero: $\lim_{J\rightarrow J_0}\limsup_{n\rightarrow\infty}\|\fE r_n^J\|_q = 0$.
\end{enumerate}
\end{proposition}

We now have what we need to prove our main result.

\begin{proof}[Proof of Theorem \ref{thm_main}]
Let $\{f_n\}$ be an $L^p$-normalized maximizing sequence. By Proposition \ref{freq loc}, after applying a symmetry, we have
\begin{align*}
\|\fE f_n^R\|_q \geq {\bf A}_p - \varepsilon(n,R),
\end{align*}
where $f_n^R := f_n\chi_{\{|\boldsymbol\xi|\leq R, |f_n|\leq R\}}$ and $\lim_{R\rightarrow\infty}\limsup_{n\rightarrow\infty}\varepsilon(n,R) = 0$. We will consider truncations of the form $f_n^m$ with $m \in \N$. Proposition \ref{L^p profile decomposition} states that, after passing to a subsequence in $n$, we may decompose $f_n^m$ as
\begin{align*}
f_n^m = \sum_{j=1}^J e^{i(t_n^{m,j},\boldsymbol x_n^{m,j})(\xi_1\xi_2,\boldsymbol\xi)}\phi^{m,j} + r_n^{m,J}, \quad\quad 1\leq J < J_0,
\end{align*}
with the decomposition satisfying properties (i)--(v) in the statement of that proposition.

Using properties (iii) and (v), then (ii), and the fact that $q > \tilde{p}$, we obtain
\begin{align*}
{\bf A}_p^q - o_m(1) &\leq \limsup_{n\rightarrow\infty}\|\fE f_n^m\|_q^q = \sum_{j=1}^{J_0}\|\fE\phi^{m,j}\|_q^q \leq {\bf A}_p^{\tilde{p}}\max_{j\leq J_0}\|\fE\phi^{m,j}\|_q^{q-\tilde{p}}\sum_{j=1}^{J_0}\|\phi^{m,j}\|_p^{\tilde{p}}\\
&\leq {\bf A}_p^{\tilde{p}}\max_{j\leq J_0}\|\fE\phi^{m,j}\|_q^{q-\tilde{p}} \leq {\bf A}_p^q\max_{j\leq J_0}\|\phi^{m,j}\|_p^{q-\tilde{p}}.
\end{align*}
Let $j = j_m$ be the index that maximizes $\|\fE\phi^{m,j}\|_q$, and set $\phi^m := \phi^{m,j_m}$ and $(t_n^m,\boldsymbol x_n^m) := (t_n^{m,j_m},\boldsymbol x_n^{m,j_m})$.
Then
\begin{align}\label{profile bounds}
1-o_m(1) \leq \|\phi^m\|_p \leq 1 \quad\quad\text{and}\quad\quad {\bf A}_p-o_m(1) \leq \|\fE\phi^m\|_p.
\end{align}
Since $e^{-i(t_n^m,\boldsymbol x_n^m)(\xi_1\xi_2,\boldsymbol\xi)}f_n^m \rightharpoonup \phi^m$ weakly in $L^p$ and $$\|\phi^m\|_p \geq (1-o_m(1))\limsup_{n\rightarrow\infty}\|f_n^m\|_p,$$ uniform convexity of $L^p$ implies that\footnote{See the proof of  \cite[Theorem 2.11]{LiebLoss}.}
\begin{align*}
\limsup_{n\rightarrow\infty}\|f_n^m-e^{i(t_n^m,\boldsymbol x_n^m)(\xi_1\xi_2,\boldsymbol\xi)}\phi^m\|_p = o_m(1).
\end{align*}
 By Proposition \ref{freq loc}, we also have
\begin{align}\label{f_n minus profile}
\limsup_{n\rightarrow\infty}\|f_n-e^{i(t_n^m,\boldsymbol x_n^m)(\xi_1\xi_2,\boldsymbol\xi)}\phi^m\|_p = o_m(1),
\end{align}
and thus
\begin{align}\label{difference}
\limsup_{n\rightarrow\infty}\|e^{i(t_n^m,\boldsymbol x_n^m)(\xi_1\xi_2,\boldsymbol\xi)}\phi^m - e^{i(t_n^{m'},\boldsymbol x_n^{m'})(\xi_1\xi_2,\boldsymbol\xi)}\phi^{m'}\|_p = o_{\min\{m,m'\}}(1).
\end{align}
Now consider the operator defined by
{\[\pi_n^mf(\boldsymbol\xi) := e^{i\theta_n^m(\boldsymbol\xi)}[\psi\ast(\varphi e^{-i\theta_n^m }f)](\boldsymbol\xi),\]
where $\theta_n^m(\boldsymbol\eta):=(t_n^m,\boldsymbol x_n^m)(\eta_1\eta_2,\boldsymbol\eta)$},
and $\varphi$ and $\psi$ are smooth, compactly supported functions on $\R^2$ with $0 \leq \varphi,\psi \leq 1$ and $\varphi(0) = \int\psi = 1$ and $\|\psi \ast (\varphi\phi^m) - \phi^m\|_p \leq \frac{1}{2}$. This operator satisfies the following properties (see display (4.2) in \cite{BS-parab-ext-exist}):
\begin{enumerate}
\item $\|\pi_n^m\|_{L^p \rightarrow L^p} \leq 1$;
\item If $|(t_n^m-t_n^{m'},\boldsymbol x_n^m-\boldsymbol x_n^{m'})|\rightarrow\infty$, then $\lim_{n\rightarrow\infty}\|\pi_n^m(e^{i\theta_n^{m'}(\boldsymbol\xi)}\phi^{m'})\|_p = 0$.
\end{enumerate}
Applying $\pi_n^m$ to \eqref{difference} and using (1), (2), and \eqref{profile bounds}, it follows that $|(t_n^m-t_n^{m'}, \boldsymbol x_n^m-\boldsymbol x_n^{m'})|$ is bounded for any $m,m' \geq M$ sufficiently large. Applying a frequency modulation to $f_n$ in \eqref{f_n minus profile}, we may assume that $(t_n^M, \boldsymbol x_n^M) \equiv 0$. Thus, for $m\geq M$, the sequence $(t_n^m,\boldsymbol x_n^m)$ is bounded. Passing to a subsequence, we may assume that $(t_n^m,\boldsymbol x_n^m) \rightarrow (t^m,\boldsymbol x^m)$ for all $m \geq M$. Consequently, by replacing $\phi^m$ with $e^{i(t^m,\boldsymbol x^m)(\xi_1\xi_2,\boldsymbol\xi)}\phi^m$, we may assume that $(t_n^m,\boldsymbol x^m) \rightarrow 0$ for all $m \geq M$. Hence, by \eqref{f_n minus profile}, the triangle inequality, and the dominated convergence theorem,
\begin{align*}
\limsup_{n\rightarrow\infty}\|f_n-\phi^m\|_p = o_m(1).
\end{align*}
This implies that $\{f_n\}$ is Cauchy and thus convergent in $L^p$. Since $\{f_n\}$ is a maximizing sequence, its limit is a maximizer.
\end{proof}


\section*{Acknowledgements}
DOS was funded by FCT/Portugal and the Recovery and Resilience Plan (PRR) through projects UID/04459/2025 and UID/PRR/04459/2025, and by project 10.54499/2023.17881.ICDT (SHADE).
BS was supported by a grant from the National Science Foundation, NSF DMS-2246906 and by a grant from the Simons Foundation [SFI-MPS-SFM-00011865, BS].


\end{document}